\documentclass[11pt]{amsart}%
\usepackage{graphicx}
\usepackage{amscd}
\usepackage{amsmath}
\usepackage{amsfonts}
\usepackage{amssymb}%
\providecommand{\U}[1]{\protect\rule{.1in}{.1in}}
\newtheorem{theorem}{Theorem}
\theoremstyle{plain}

\newtheorem{corollary}{Corollary}[section]

\newtheorem{definition}{Definition}[section]

\newtheorem{lemma}{Lemma}[section]

\newtheorem{remark}{Remark}[section]

\numberwithin{equation}{section}
\numberwithin{theorem}{section}

\begin{document}
\title[Classical and Sobolev Orthogonality of the Nonclassical Jacobi Polynomials]{Classical and Sobolev Orthogonality of the Nonclassical Jacobi Polynomials
with Parameters $\alpha=\beta=-1$}
\author{Andrea Bruder}
\address{Department of Mathematics and Computer Science, Colorado College, Tutt Science
Center, 14 E. Cache la Poudre St., Colorado Springs, CO 80903}
\email{Andrea.Bruder@ColoradoCollege.edu}
\author{Lance L. Littlejohn}
\address{Department of Mathematics, Baylor University, One Bear Place \#97328, Waco, TX 76798-7328}
\email{Lance\_Littlejohn@baylor.edu}
\date{May 22, 2012 (BL2v4.tex)}
\dedicatory{We dedicate this paper to the memory of our teacher, mentor, and friend
Professor W. N. Everitt (1924-2011)}

\begin{abstract}
In this paper, we consider the second-order differential expression%
\[
\ell\lbrack y](x)=(1-x^{2})(-(y^{\prime}(x))^{\prime}+k(1-x^{2})^{-1}%
y(x))\quad(x\in(-1,1)).
\]
This is the Jacobi differential expression with non-classical parameters
$\alpha=\beta=-1$ in contrast to the classical case when $\alpha,\beta>-1$.
\ For fixed $k\geq0$ and appropriate values of the spectral parameter
$\lambda,$ the equation $\ell\lbrack y]=\lambda y$ has, as in the classical
case, a sequence of (Jacobi) polynomial solutions $\{P_{n}^{(-1,-1)}%
\}_{n=0}^{\infty}.$ These Jacobi polynomial solutions of degree $\geq2$ form a
complete orthogonal set in the Hilbert space $L^{2}((-1,1);(1-x^{2})^{-1}).$
Unlike the classical situation, \textit{every} polynomial of degree one is a
solution of this eigenvalue equation. Kwon and Littlejohn showed that, by
careful selection of this first degree solution, the set of polynomial
solutions of degree $\geq0$ are orthogonal with respect to a Sobolev inner
product. Our main result in this paper is to construct a self-adjoint operator
$T,$ generated by $\ell\lbrack\cdot],$ in this Sobolev space that has these
Jacobi polynomials as a complete orthogonal set of eigenfunctions. The
classical Glazman-Krein-Naimark theory is essential in helping to construct
$T$ in this Sobolev space as is the left-definite theory developed by
Littlejohn and Wellman.

\end{abstract}
\maketitle

\section{Introduction\label{Introduction}}

For $\alpha,\beta>-1,$ the spectral properties of the classical Jacobi
differential expression%
\begin{align}
\ell_{\alpha,\beta}[y](x)  &  :=\frac{1}{\omega_{\alpha,\beta}(x)}\left[
\left(  \left(  -(1-x)^{\alpha+1}(1+x)^{\beta+1}\right)  y^{\prime}(x)\right)
^{\prime}+k(1-x)^{\alpha}(1+x)^{\beta}y(x)\right] \label{Jacobi alpha,beta}\\
&  =-(1-x^{2})y^{\prime\prime}(x)+(\alpha-\beta+(\alpha+\beta+2)x)y^{\prime
}(x)+ky(x)\nonumber
\end{align}
where $k\geq0$ is fixed, $x\in(-1,1)$ and
\begin{equation}
\omega_{\alpha,\beta}(x):=(1-x)^{\alpha}(1+x)^{\beta}, \label{Jacobi weight}%
\end{equation}
are well understood. In this case, the $n^{th}$ degree Jacobi polynomial
$y=P_{n}^{(\alpha,\beta)}(x)$ is a solution of the equation%
\begin{equation}
\ell_{\alpha,\beta}[y](x)=\left(  n(n+\alpha+\beta+1)+k\right)  y(x)\qquad
(n\in%
\mathbb{N}
_{0}); \label{Jacobi Eigenvalue Equation}%
\end{equation}
details of the properties of these polynomials can be found in the classic
texts \cite{Chihara} and \cite{Szego}. The right-definite spectral analysis
has been studied in \cite{Akhiezer-Glazman} and \cite{EKLWY}. Through the
Glazman-Krein-Naimark (GKN) theory \cite{Naimark}, it has been known that
there exists a self-adjoint operator $A^{(\alpha,\beta)},$ generated from the
Jacobi differential expression (\ref{Jacobi alpha,beta}), in the Hilbert space
$L^{2}((-1,1);w_{\alpha,\beta})$ of Lebesgue measurable functions
$f:(-1,1)\rightarrow\mathbb{C}$ satisfying%
\[
\left\Vert f\right\Vert _{\alpha,\beta}^{2}:=\int_{-1}^{1}\left\vert
f(x)\right\vert ^{2}w_{\alpha,\beta}(x)dx<\infty
\]
which has the Jacobi polynomials as a complete set of eigenfunctions.

For $\alpha,\beta\geq-1,$ let
\begin{equation}
L_{\alpha,\beta}^{2}(-1,1):=L^{2}((-1,1);w_{\alpha,\beta}) \label{L^2 spaces}%
\end{equation}
be the weighted Hilbert space with usual inner product%
\begin{equation}
(f,g)_{\alpha,\beta}=\int_{-1}^{1}f(x)\overline{g}(x)w_{\alpha,\beta}(x)dx
\label{Jacobi IP}%
\end{equation}
and related norm $\left\Vert \cdot\right\Vert _{\alpha,\beta}=(\cdot
,\cdot)_{\alpha,\beta}^{1/2}.$ In this paper, we will study the spectral
theory of the Jacobi expression (\ref{Jacobi alpha,beta}) in $L_{-1,-1}%
^{2}(-1,1)$ (that is when $\alpha=\beta=-1)$ as well as in several other
Hilbert spaces in which the associated Jacobi polynomials are orthogonal.

For $\alpha,\beta>-1$ and $n\in\mathbb{N}_{0,}$ the Jacobi polynomial
$P_{n}^{(\alpha,\beta)}(x)$ of degree $n$ is defined (see \cite{Chihara},
\cite{Rainville}, \cite[Chapter IV]{Szego}) to be any non-zero multiple of
\begin{equation}
P_{n}^{(\alpha,\beta)}(x):=\sum\limits_{j=0}^{n}\binom{n+\alpha}{j}%
\binom{n+\beta}{n-j}\left(  \dfrac{x-1}{2}\right)  ^{j}\left(  \dfrac{x+1}%
{2}\right)  ^{n-j}; \label{Jacobi polynomial alpha, beta}%
\end{equation}
it is well known, in this case, that the Jacobi polynomials $\{P_{n}%
^{(\alpha,\beta)}\}_{n=0}^{\infty}$ form a complete orthogonal set in
$L_{\alpha,\beta}^{2}(-1,1).$ From (\ref{Jacobi polynomial alpha, beta}),
notice that
\[
P_{0}^{(\alpha,\beta)}(x)=1\text{ and }P_{1}^{(\alpha,\beta)}(x)=\alpha
+1+(\alpha+\beta+2)\frac{(1-x)}{2}%
\]
so, in particular from this definition, we see that $P_{1}^{(-1-1)}%
(x)\equiv0.$ Furthermore,
\[
P_{0}^{(-1,-1)}(x)\notin L_{-1,-1}^{2}(-1,1)
\]
because of the singularities in the weight function $w_{-1,-1}(x)=(1-x^{2}%
)^{-1}.$ However, as we will see, it is the case that $\{P_{n}^{(-1,-1)}%
\}_{n=2}^{\infty}$ does form a complete orthogonal set in $L_{-1,-1}%
^{2}(-1,1).$ In this weighted Hilbert space, we will apply the
Glazman-Krein-Naimark (GKN) theory \cite[Chapter IV]{Naimark} to construct the
(unique) self-adjoint operator $A=A^{(-1,-1)},$ generated by $\ell
_{-1,-1}[\cdot],$ having $\{P_{n}^{(-1,-1)}\}_{n=2}^{\infty}$ as
eigenfunctions. As the reader will see, this operator $A$ will be key to
subsequent analysis that we develop.

When $\alpha=\beta=-1,$ every first degree polynomial is a solution of
(\ref{Jacobi Eigenvalue Equation})$.$ Kwon and Littlejohn
\cite{Kwon-Littlejohn} showed that, by careful choice of this first degree
polynomial, the entire sequence $\{P_{n}^{(-1,-1)}\}_{n=0}^{\infty}$ forms an
orthogonal set in a certain Sobolev space $W$ generated by the Sobolev inner
product
\begin{equation}
\phi\left(  f,g\right)  :=\frac{1}{2}f(-1)\overline{g}(-1)+\frac{1}%
{2}f(1)\overline{g}(1)+\int_{-1}^{1}f^{\prime}(t)\overline{g}^{\prime}(t)dt.
\label{Sobolev IP}%
\end{equation}
Moreover, in fact, these polynomials form a complete orthogonal set in $W.$ We
note that, by Favard's Theorem, the entire set $\{P_{n}^{(-1,-1)}%
\}_{n=0}^{\infty},$ for any choice of the first degree polynomial
$P_{1}^{(-1,-1)}(x),$ cannot be orthogonal on the real line with respect to a
measure, signed or otherwise.

The main part of this paper is, however, to construct a self-adjoint operator
$T$, generated by $\ell\lbrack\cdot],$ in $W$ that has the Jacobi polynomials
$\{P_{n}^{(-1,-1)}\}_{n=0}^{\infty}$ as eigenfunctions. The GKN theory, as
well as the general left-definite operator theory developed by Littlejohn and
Wellman \cite{Littlejohn-Wellman}, is of paramount importance in the
construction of this self-adjoint operator.

We note that, for $m\in\mathbb{N},$ the Jacobi polynomials $\{P_{n}%
^{(\alpha,-m)}\}_{n=0}^{\infty}$ are orthogonal with respect to inner products
of the form (\ref{Sobolev IP}) but whose integrand involves the $m^{th}$
derivative of the functions. In this respect, we refer the reader to
\cite{Alfaro et al}, \cite{Alfaro et al II}, \cite{Alvarez},
\cite{Kwon-Lee-Littlejohn}, and \cite{Kwon-Littlejohn} where general results
on the Sobolev orthogonality of the Jacobi or Gegenbauer polynomials, when one
or both parameters $\alpha$ and $\beta$ are negative integers, are obtained.
Bruder and Littlejohn \cite{BruderLittlejohn} developed the spectral theory
when $\alpha>-1$ and $\beta=-1.$ The analysis in \cite{BruderLittlejohn} is
similar in some respects to some of the results of this paper but, overall,
quite different; whenever possible, we omit proofs which are similar to those
given in \cite{BruderLittlejohn}.

The contents of this paper are as follows. In Section \ref{Preliminaries}, we
discuss several well known properties of the Jacobi polynomials that will be
useful for subsequent analysis. Section \ref{Operator Inequality} deals with
an important operator inequality \cite{Chisholm-Everitt-Littlejohn} that is
essential for much of the hard analytic results that we develop. The classical
GKN theory is used in Section \ref{Right Definite Analysis} to construct the
self-adjoint operator $A$ in $L^{2}((-1,1);(1-x^{2})^{-1});$ in this section,
we also prove several properties of functions in the domain $\mathcal{D}(A)$
of $A$ that are necessary later in the paper. A short review of the general
left-definite theory developed by Littlejohn and Wellman is given in Section
\ref{Left-definite theory}. It is remarkable that this theory is important in
developing the spectral theory of the Jacobi expression $\ell\lbrack\cdot]$ in
the Sobolev space $W,$ whose properties we develop in Section
\ref{Sobolev Spectral Theory}. In Section \ref{Left-definite Jacobi}, the
left-definite theory of $\ell\lbrack\cdot]$ is developed. Results from this
section are then used to construct the self-adjoint operator $T$ in $W$ in
Section \ref{Sobolev Spectral Theory}.

\section{Preliminaries: Classical Properties of the Jacobi
Polynomials\label{Preliminaries}}

For $\alpha,\beta>-1,$ it is convenient for later purposes to define the
Jacobi polynomials $\{P_{n}^{(\alpha,\beta)}\}_{n=0}^{\infty}$ as%
\begin{equation}
P_{n}^{(\alpha,\beta)}(x):=k_{n}^{\alpha,\beta}\sum\limits_{j=0}^{n}%
\frac{(1+\alpha)_{n}(1+\alpha+\beta)_{n+j}}{j!(n-j)!(1+\alpha)_{j}%
(1+\alpha+\beta)_{n}}\left(  \frac{1-x}{2}\right)  ^{j},
\label{Jacobi polynomial definition a,b}%
\end{equation}

\noindent where
\[
k_{n}^{\alpha,\beta}:=\frac{(n!)^{1/2}(1+\alpha+\beta+2n)^{1/2}(\Gamma
(\alpha+\beta+n+1))^{1/2}}{2^{(\alpha+\beta+1)/2}(\Gamma(\alpha+n+1))^{1/2}%
(\Gamma(\beta+n+1))^{1/2}}.
\]
With this normalization, these Jacobi polynomials form a complete
\textit{orthonormal} set in $L_{\alpha,\beta}^{2}(-1,1).$

Their derivatives satisfy the identity%
\begin{equation}
\frac{d^{j}}{dx^{j}}P_{n}^{(\alpha,\beta)}(x)=a^{(\alpha,\beta)}%
(n,j)P_{n-j}^{(\alpha+j,\beta+j)}(x)\quad(n,j\in%
\mathbb{N}
_{0}), \label{(1) Colorado}%
\end{equation}
where%
\[
a^{(\alpha,\beta)}(n,j)=\frac{(n!)^{1/2}\left(  \Gamma(\alpha+\beta
+n+1+j)\right)  ^{1/2}}{((n-j)!)^{1/2}\left(  \Gamma(\alpha+\beta+n+1)\right)
^{1/2}}\quad(j=0,1,...,n),
\]
and $a^{(\alpha,\beta)}(n,j)=0$ if $j>n$. Furthermore, for $n,r,j\in
\mathbb{N}_{0},$ we have the orthogonality relation%
\begin{equation}
\int_{-1}^{1}\frac{d^{j}\left(  P_{n}^{(\alpha,\beta)}(x)\right)  }{dx^{j}%
}\frac{d^{j}\left(  P_{r}^{(\alpha,\beta)}(x)\right)  }{dx^{j}}w_{\alpha
+j,\beta+j}(x)dx=\frac{n!\Gamma(\alpha+\beta+n+1+j)}{(n-j)!\Gamma(\alpha
+\beta+n+1)}\delta_{n,r}. \label{(4) Colorado}%
\end{equation}

For $\alpha=\beta=-1,$ the `natural' setting for an analytical study of the
Jacobi polynomials is the Hilbert space $L_{-1,-1}^{2}(-1,1).$ However,
because of the singularities in the associated orthogonalizing weight function
$w_{-1,-1}(x)=(1-x^{2})^{-1}$, only the Jacobi polynomials $P_{n}%
^{(-1,-1)}(x)$ of degree $n\geq2$ belong to this space; see Section
\ref{Sobolev Spectral Theory}.

The following result is well known and can be found, for example, in
\cite[Chapter III, Theorem 3.1.5]{Szego}.

\begin{lemma}
\label{complete jacobi -1}The sequence $\left\{  P_{n}^{(1,1)}(x)\right\}
_{n=0}^{\infty}$ forms a complete orthogonal set in the Hilbert space
$L_{1,1}^{2}(-1,1).$
\end{lemma}

From this lemma, we prove the following result.

\begin{lemma}
\label{completeness -1-1}The sequence $\left\{  P_{n}^{(-1,-1)}(x)\right\}
_{n=2}^{\infty}$ forms a complete orthogonal set in the Hilbert space
$L_{-1,-1}^{2}(-1,1).$ Equivalently, the set of all polynomials $\mathcal{P}%
_{-1}[-1,1]$ of degree $\geq2$ satisfying $p(\pm1)=0$ is dense in
$L_{-1,-1}^{2}(-1,1)$. In particular,
\begin{equation}
P_{n}^{(-1,-1)}(\pm1)=0\text{\quad}(n\geq2).\label{Value at -1 and +1}%
\end{equation}

\end{lemma}

\begin{proof}
Note that%
\[
\int_{-1}^{1}\left\vert f(x)\right\vert ^{2}(1-x^{2})^{-1}dx=\int_{-1}%
^{1}\left\vert (1-x^{2})^{-1}f(x)\right\vert ^{2}(1-x^{2})dx,
\]
i.e. $f\in L^{2}\left(  (-1,1);(1-x^{2})^{-1}\right)  \iff(1-x^{2})^{-1}f\in
L^{2}\left(  (-1,1);(1-x^{2})\right)  ,$ and in this case,%
\begin{equation}
\left\Vert f\right\Vert _{-1,-1}=\left\Vert (1-x^{2})^{-1}f\right\Vert
_{1,1}.\label{Completeness -1-1 1}%
\end{equation}
Let $f\in L^{2}\left(  (-1,1);(1-x^{2})^{-1}\right)  ,$ and let $\epsilon>0.$
Hence%
\[
(1-x^{2})^{-1}f\in L^{2}\left(  (-1,1);(1-x^{2})\right)  ,
\]
so by Lemma \ref{complete jacobi -1}, there exists a polynomial $q(x)\ $with%
\[
\left\Vert (1-x^{2})^{-1}f-q\right\Vert _{1,1}<\epsilon.
\]
Let $p(x)=(1-x^{2})q(x)$ so $q(x)=(1-x^{2})^{-1}p(x).$ Then $p$ is a
polynomial of degree $\geq2,$ with $p(\pm1)=0,$ and%
\begin{align*}
\epsilon &  >\left\Vert (1-x^{2})^{-1}f-(1-x^{2})^{-1}p\right\Vert _{1,1}\\
&  =\left\Vert (1-x^{2})^{-1}(f-p)\right\Vert _{1,1}\\
&  =\left\Vert f-p\right\Vert _{-1,-1}\text{ by (\ref{Completeness -1-1 1}).}%
\end{align*}
Lastly, to establish $($\ref{Value at -1 and +1}$),$ we note that it is
straightforward to see, from the definition in $($%
\ref{Jacobi polynomial alpha, beta}$),$ that $P_{n}^{(-1,-1)}(x)$ is a
non-zero multiple of $(1-x^{2})P_{n-2}^{(1,1)}(x);$ details can be found in
\cite[Lemma 5.4]{Bruder}. We note that there is a similar situation with the
Laguerre polynomials $\{L_{n}^{\alpha}\}$ at $x=0$ when $\alpha$ is a negative
integer; see \cite[p. 102]{Szego}.
\end{proof}

\section{An Operator Inequality\label{Operator Inequality}}

The following result, due to Chisholm and Everitt \cite{Chisholm-Everitt}, is
important in establishing the main analytic results in this paper$;$ it has
been a remarkably useful tool in obtaining general properties of functions in
certain operator domains. Theorem \ref{CHEL} was extended to the general case
of conjugate indices $p$ and $q$ $(p,q>1)$ in
\cite{Chisholm-Everitt-Littlejohn} in 1999. Several years after publication,
the authors of \cite{Chisholm-Everitt-Littlejohn} learned that this result was
first established by Talenti \cite{Talenti} and Tomaselli \cite{Tomaselli},
both in 1969, and later by Muckenhoupt \cite{Muckenhoupt} in 1972.

\begin{theorem}
\label{CHEL}Suppose $I=(a,b)$ is an open interval of the real line, where
$-\infty\leq a<b\leq\infty.$ Suppose $w$ is a positive Lebesgue measurable
function on $(a,b)$ and $\varphi,\psi$ are functions satisfying the three conditions:

\begin{enumerate}
\item[(i)] $\varphi\in L_{\mathrm{loc}}^{2}(\,(a,b);w)$ and $\psi\in
L_{\mathrm{loc}}^{2}(\,(a,b);w);$

\item[(ii)] for some $c\in(a,b),$ $\varphi\in L^{2}(\,(a,c];w)$ and $\psi\in
L^{2}(\,[c,b);w);$

\item[(iii)] for all $[\alpha,\beta]\subset(a,b),$%
\[
\int_{a}^{\alpha}\left\vert \varphi(t)\right\vert ^{2}w(t)dt\,>0\text{ and
}\int_{\beta}^{b}\left\vert \psi(t)\right\vert ^{2}w(t)dt\,>0.
\]

\end{enumerate}

\noindent Define the linear operators $A$ and $B$ on $L^{2}(\,(a,b);w)$ and
$L^{2}(\,(a,b);w)$, respectively, by%
\[
(Ag)(x):=\varphi(x)\int_{x}^{b}\psi(t)g(t)w(t)dx\quad(\,x\in(a,b)\text{ and
}g\in L^{2}(\,(a,b);w)\,)
\]%
\[
(Bg)(x):=\psi(x)\int_{a}^{x}\varphi(t)g(t)w(t)dx\quad(\,x\in(a,b)\text{ and
}g\in L^{2}(\,(a,b);w)\,);
\]
then%
\begin{align*}
A  &  :L^{2}(\,(a,b);w)\rightarrow L_{l\mathrm{oc}}^{2}(\,(a,b):w)\\
B  &  :L^{2}(\,(a,b);w)\rightarrow L_{\mathrm{loc}}^{2}(\,(a,b);w).
\end{align*}

\noindent Define $K(\cdot):$ $(a,b)\rightarrow(0,\infty)$ by%
\[
K(x):=\left\{  \int_{a}^{x}\left\vert \varphi(t)\right\vert ^{2}%
w(t)dx\right\}  ^{1/2}\left\{  \int_{x}^{b}\left\vert \psi(t)\right\vert
^{2}w(t)dt\right\}  ^{1/2}\quad(\,x\in(a,b)\,)
\]
and the number $K\in(0,\infty]$%
\[
K:=\sup\{\,K(x)\mid x\in(a,b)\,\}.
\]

\noindent Then a necessary and sufficient condition that $A$ and $B$ are
bounded linear operators on $L^{2}(\,(a,b);w)$ is that the number $K$ is
finite, \textit{i.e.}%
\[
K\in(0,\infty).
\]
Furthermore, the following operator inequalities are valid:
\[
\left\Vert Af\right\Vert _{2}\leq2K\left\Vert f\right\Vert _{2}\quad(\,f\in
L^{2}(\,(a,b);w)\,)
\]%
\[
\left\Vert Bg\right\Vert _{2}\leq2K\left\Vert g\right\Vert _{2}\quad(\,g\in
L^{2}(\,(a,b);w)\,);
\]
the number $2K$ given in the above inequalities is best possible for these
inequalities to hold.
\end{theorem}

\section{Right-Definite Spectral Analysis of the Jacobi Expression when
$\alpha=\beta=-1$\label{Right Definite Analysis}}

In the special case when $\alpha=\beta=-1,$ the Jacobi differential expression
(\ref{Jacobi alpha,beta}) simplifies to be%
\begin{align}
\ell\lbrack y](x):=\ell_{-1,-1}[y](x)  &  =(1-x^{2})\left(  \left(
-(y^{\prime}(x)^{\prime}+k(1-x^{2})^{-1}y(x\right)  \right)  \label{Jacobi DE}%
\\
&  =-(1-x^{2})y^{\prime\prime}(x)+ky(x)\quad(x\in(-1,1));\nonumber
\end{align}
here $k$ is a fixed, non-negative constant. The maximal domain associated with
$\ell\lbrack\cdot]$ in $L_{-1,-1}^{2}(-1,1)$ is%
\[
\Delta:=\left\{  f:(-1,1)\rightarrow%
\mathbb{C}
\mid f,f^{\prime}\in AC_{\mathrm{loc}}(-1,1);f,\ell\lbrack f]\in L_{-1,-1}%
^{2}(-1,1)\right\}  .
\]
Observe that if $f\in\Delta,$ then $(1-x^{2})f^{\prime\prime}\in$
$L_{-1,-1}^{2}(-1,1)$ or, equivalently%
\begin{equation}
(1-x^{2})^{1/2}f^{\prime\prime}\in L^{2}(-1,1)\quad(f\in\Delta);
\label{Maximal Domain 1}%
\end{equation}
moreover, $f\in L_{-1,-1}^{2}(-1,1)$ is equivalent to
\begin{equation}
(1-x^{2})^{-1/2}f\in L^{2}(-1,1)\quad(f\in\Delta). \label{Maximal Domain 2}%
\end{equation}
For $f,g\in\Delta$ and $[a,b]\subset(-1,1)$, we have \textit{Dirichlet's
formula}%
\begin{equation}
\int_{a}^{b}\ell\lbrack f](x)\overline{g}(x)(1-x^{2})^{-1}dx=-f^{\prime
}(x)\overline{g}(x)\mid_{a}^{b}+\int_{a}^{b}\left[  f^{\prime}(x)\overline
{g}^{\prime}(x)+k(1-x^{2})^{-1}f(x)\overline{g}(x)\right]  dx
\label{Dirichlet's formula}%
\end{equation}
and \textit{Green's formula}%
\begin{equation}
\int_{a}^{b}\ell\lbrack f](x)\overline{g}(x)(1-x^{2})^{-1}dx=\left[
f(x)\overline{g}^{\prime}(x)-f^{\prime}(x)\overline{g}(x)\right]  \mid_{a}%
^{b}+\int_{a}^{b}f(x)\overline{\ell\lbrack g]}(x)(1-x^{2})^{-1}dx\text{.}
\label{Green's formula}%
\end{equation}

\begin{theorem}
\label{SLP and D}The Jacobi differential expression $($\ref{Jacobi DE}$)$\ is
strong limit-point $($SLP$)$ and Dirichlet at $x=\pm1$. That is to say, for
$f,g\in\Delta$
\end{theorem}

\begin{enumerate}
\item[(i)] (Dirichlet) $\int_{0}^{1}\left\vert f^{\prime}(t)\right\vert
^{2}dt<\infty$ and $\int_{-1}^{0}\left\vert f^{\prime}(t)\right\vert
^{2}dt<\infty$,

\item[(ii)] (SLP) $\lim\limits_{x\rightarrow\pm1}f^{\prime}(x)\overline
{g}(x)=0.$
\end{enumerate}

The proof of this theorem will follow immediately after the following three
lemmas are established.

\begin{lemma}
[Dirichlet Condtion]\label{Dirichlet -1,-1}For $f\in\Delta,f^{\prime}\in
L^{2}(-1,1).$ In particular, we may assume that $f\in AC[-1,1]$ for all
$f\in\Delta.$
\end{lemma}

\begin{proof}
We prove that $f^{\prime}\in L^{2}[0,1);$ a similar proof establishes
$f^{\prime}\in L^{2}(-1,0].$ Since $f^{\prime}\in AC_{\mathrm{loc}}[0,1),$ we
see that%
\begin{equation}
f^{\prime}(x)=f(0)+\int_{0}^{x}\frac{f^{\prime\prime}(t)\sqrt{1-t^{2}}}%
{\sqrt{1-t^{2}}}dt\text{ \ \ \ \ \ \ }\left(  x\in\lbrack0,1)\right)  .
\label{Dirichlet -1 -1 1}%
\end{equation}
We now apply Theorem \ref{CHEL} with $\varphi(x)=1$ and $\psi(x)=1/\sqrt
{1-x^{2}}.$ Since%
\[
\int_{0}^{x}\psi^{2}(t)dt\cdot\int_{x}^{1}\varphi^{2}(t)dt=\frac{1}{2}%
(1-x)\ln(\frac{1+x}{1-x})
\]
is bounded on $[0,1)$, we see that $\int\limits_{0}^{x}\frac{f^{\prime\prime
}(t)\sqrt{1-t^{2}}}{\sqrt{1-t^{2}}}dt\in L^{2}[0,1).$ Hence, from
(\ref{Dirichlet -1 -1 1}), $f^{\prime}\in L^{2}[0,1).$
\end{proof}

\begin{lemma}
\label{f(plusminus 1)=0}For all $f\in\Delta,$ $f(\pm1)=0.$
\end{lemma}

\begin{proof}
Note that from the previous lemma, $f\in AC[-1,1]$ and thus the limits
\[
f(\pm1):=\lim_{x\rightarrow\pm1}f(x)
\]
exist and are finite. Suppose that $f(1)\neq0;$ we can, without loss of
generality, assume that $f(1)>0.$ By continuity, there exists $x^{\ast}%
\in(0,1)$ such that%
\[
f(x)>\frac{f(1)}{2}\text{ \ \ \ \ \ \ for }x\in\lbrack x^{\ast},1).
\]
Then
\[
\infty>\int_{0}^{1}\left\vert f(t)\right\vert ^{2}(1-t^{2})^{-1}dt\geq
\int_{x^{\ast}}^{1}\left\vert f(t)\right\vert ^{2}(1-t^{2})^{-1}dt\geq
\frac{\left(  f(1)\right)  ^{2}}{4}\int_{x^{\ast}}^{1}(1-t^{2})^{-1}%
dt=\infty,
\]
a contradiction. A similar argument shows that $f(-1)=0.$
\end{proof}

\begin{lemma}
[Strong Limit-Point Condition]For all $f,g\in\Delta,\lim\limits_{x\rightarrow
\pm1^{\mp}}f(x)g^{\prime}(x)=0.$
\end{lemma}

\begin{proof}
Let $f,g\in\Delta.$ It suffices to prove that
\[
\lim_{x\rightarrow1^{-}}f(x)g^{\prime}(x)=0;
\]
a similar argument establishes the other limit. We assume that $f$ and $g$ are
both real-valued. Note, by H\"{o}lder's inequality, (\ref{Maximal Domain 1}),
and (\ref{Maximal Domain 2}) that $fg^{\prime\prime}\in L^{1}(-1,1)$ so that
$\lim\limits_{x\rightarrow1^{-}}\int_{0}^{x}f(t)g^{\prime\prime}(t)dt$ exists
and is finite. Now, by integration by parts,
\[
\int_{0}^{x}f(t)g^{\prime\prime}(t)dt=f(t)g^{\prime}(t)\mid_{0}^{x}-\int%
_{0}^{x}f^{\prime}(t)g^{\prime}(t)dt.
\]
By Lemma \ref{Dirichlet -1,-1}, $\lim\limits_{x\rightarrow1^{-}}\int_{0}%
^{x}f^{\prime}(t)g^{\prime}(t)dt$ exists and is finite. Hence, we see that
$\lim\limits_{x\rightarrow1^{-}}f(x)g^{\prime}(x)$ exists and is finite.
Suppose that $\lim\limits_{x\rightarrow1^{-}}f(x)g^{\prime}(x)=c\neq0;$ we may
assume that $c>0.$ For $x$ close to $1,$ we may also assume that
\[
f(x)>0\text{ and }g^{\prime}(x)>0.
\]
Hence, there exists $x^{\ast}\in\lbrack0,1)$ such that $g^{\prime}%
(x)\geq\dfrac{c}{2f(x)}$ for $x\in\lbrack x^{\ast},1).$ Therefore,%
\[
\left\vert f^{\prime}(x)g^{\prime}(x)\right\vert \geq\frac{c\left\vert
f^{\prime}(x)\right\vert }{2f(x)}\text{ \ \ \ \ \ \ }\left(  x\in\lbrack
x^{\ast},1)\right)  .
\]
Integrate to obtain%
\[
\int_{x^{\ast}}^{x}\left\vert f^{\prime}(t)g^{\prime}(t)\right\vert
dt\geq\dfrac{c}{2}\int_{x^{\ast}}^{x}\frac{\left\vert f^{\prime}(t)\right\vert
}{f(t)}dt\geq\dfrac{c}{2}\left\vert \int_{x^{\ast}}^{x}\frac{f^{\prime}%
(t)}{f(t)}dt\right\vert =\dfrac{c}{2}\left\vert \ln f(x)\right\vert +\gamma,
\]
where $\gamma$ is some constant of integration. Now let $x\rightarrow1^{-};$
we see from Lemma \ref{f(plusminus 1)=0} that
\[
\infty>\int_{x^{\ast}}^{1}\left\vert f^{\prime}(t)g^{\prime}(t)\right\vert
dt\geq\dfrac{c}{2}\lim\limits_{x\rightarrow1^{-}}\left\vert \ln
f(x)\right\vert +k=\infty.
\]
This contradiction shows that $c=0$ and this establishes the lemma.
\end{proof}

We now define the operator%
\[
A:\mathcal{D}(A)\subset L_{-1,-1}^{2}(-1,1)\rightarrow L_{-1,-1}^{2}(-1,1)
\]
by%
\begin{equation}%
\begin{array}
[c]{c}%
Af=\ell\lbrack f]\\
f\in\mathcal{D}(A):=\Delta.
\end{array}
\label{Jacobi operator}%
\end{equation}
Since $x=\pm1$ are SLP, we see from the Glazman-Krein-Naimark theory
\cite[Chapter V]{Naimark} that $A$ is an unbounded, self-adjoint operator in
$L_{-1,-1}^{2}(-1,1)$ with spectrum%
\[
\sigma(A)=\{n(n-1)+k\mid n=2,3,\ldots\}
\]
Moreover, from (\ref{Dirichlet's formula}), (\ref{Green's formula}), and
Theorem \ref{SLP and D}, we have the classic Green's formula,%
\[
\left(  Af,g\right)  _{-1,-1}=\int_{-1}^{1}\ell\lbrack f](x)\overline
{g}(x)(1-x^{2})^{-1}dx=\left(  f,Ag\right)  _{L_{-1,-1}^{2}(-1,1)},
\]
and Dirichlet's formula,%
\[
\left(  Af,g\right)  _{-1,-1}=\int_{-1}^{1}\left[  f^{\prime}(x)\overline
{g}^{\prime}(x)+k(1-x^{2})^{-1}f(x)\overline{g}(x)\right]  dx.
\]
In particular, observe that%
\begin{align}
\left(  Af,f\right)  _{-1,-1}  &  =\int_{-1}^{1}\left[  \left\vert f^{\prime
}(x)\right\vert ^{2}+k(1-x^{2})^{-1}\left\vert f(x)\right\vert ^{2}\right]
dx\label{Jacobi operator bounded below}\\
&  \geq k\left(  f,f\right)  _{-1,-1};\nonumber
\end{align}
in other words, the self-adjoint operator $A$ is bounded below in
$L_{-1,-1}^{2}(-1,1)$ by $kI$ where $I$ is the identity operator. This
observation will be important from the viewpoint of left-definite theory,
which we now briefly discuss.

\section{General Left-Definite Theory\label{Left-definite theory}}

In \cite{Littlejohn-Wellman}, Littlejohn and Wellman developed a general
abstract left-definite theory for a self-adjoint operator $T$ that is bounded
below in a Hilbert space $(H,(\cdot,\cdot))$. The results contained here are
important for our construction of the Jacobi self-adjoint operator $T$ in the
Sobolev space $W,$ see Section \ref{Sobolev Spectral Theory}, having the
Jacobi polynomials $\{P_{m}^{(-1,-1)}\}_{m=0}^{\infty}$ as a complete set of eigenfunctions.

There is a strong connection between left-definite theory and the theory of
Hilbert scales; indeed, our left-definite spaces are Hilbert scales. We refer
the reader to the recent texts of Albeverio and Kurasov \cite[Chapter
1.2.2]{Albeverio and Kurasov} and Simon \cite[Chapter 12]{Simon} as well as
the classic texts of Berezanski\u{\i} \cite{Berezanski} and Maz'ya
\cite{Maz'ya}.

Let $V$ be a vector space over $%
\mathbb{C}
$ with inner product $(\cdot,\cdot)$ such that $H=(V,(\cdot,\cdot))$ is a
Hilbert space. Let $s>0$ and suppose that $V_{s}$ is a vector subspace of $V$
with inner product $(\cdot,\cdot)_{s};$ let $H_{s}=(V_{s},(\cdot,\cdot)_{s})$
denote this inner product space.

Throughout this section, we assume that $T:\mathcal{D}(T)\subset H\rightarrow
H$ is a self-adjoint operator that is bounded below by $rI$ for some $r>0$,
where $I$ is the identity operator on $H;$ that is to say%
\[
(Tx,x)\geq r(x,x)\text{\quad}(x\in\mathcal{D}(T)).
\]
\noindent It is well known that, for $s>0$, the operator $T^{s}$ is
self-adjoint and bounded below in $H$ by $r^{s}I$.

\begin{definition}
\label{left-definite space definition}We say that $H_{s}=(V_{s},(\cdot
,\cdot)_{s})$ is an $s^{th}$ left-definite space associated with the pair
$(H,A)$ if
\end{definition}

\begin{enumerate}
\item[(i)] $H_{s}$ is a Hilbert space;

\item[(ii)] $\mathcal{D}(T^{s}),$ the domain of $T^{s},$ is a vector subspace
of $V_{s};$

\item[(iii)] $\mathcal{D}(T^{s})$ is dense in $H_{s};$

\item[(iv)] $(x,x)_{s}\geq r^{s}(x,x)$ for all $x\in V_{s};$

\item[(v)] $(x,y)_{s}=(T^{s}x,y)$ for all $x\in\mathcal{D}(T^{s}),y\in
V_{s}.\medskip$
\end{enumerate}

\noindent Littlejohn and Wellman in \cite[Theorem 3.1]{Littlejohn-Wellman}
prove the following existence/uniqueness result.

\begin{theorem}
Let $T:\mathcal{D}(T)\subset H\rightarrow H$ be a self-adjoint operator that
is bounded below by $rI$ for some $r>0$. Let $s>0$ and define $H_{s}%
=(V_{s},(\cdot,\cdot)_{s})$ by
\[
V_{s}=\mathcal{D}(T^{s/2}),
\]
and%
\[
(x,y)_{s}=(T^{s/2}x,T^{s/2}y)\text{ \ \ \ \ \ \ \ \ \ \ }(x,y\in V_{s}).
\]
Then $H_{s}$ is the unique left-definite space associated with the pair
$(H,T)$.
\end{theorem}

\begin{definition}
For $s>0,$ let $H_{s}=(V_{s},(\cdot,\cdot)_{s})$ be the $s^{th}$ left-definite
space associated with $(H,T)$. If there exists a self-adjoint operator
$T_{s}:\mathcal{D}(T_{s})\subset H_{s}\rightarrow H_{s}$ satisfying%
\[
T_{s}x=Tx\text{ \ \ \ \ \ \ \ \ \ \ \ }(x\in\mathcal{D}(T_{s})\subset
\mathcal{D}(T)),
\]
then we call such an operator an $s^{th}$ left-definite operator associated
with the pair $(H,T)$.
\end{definition}

In \cite[Theorem 3.2]{Littlejohn-Wellman}, the authors establish the following
existence/uniqueness result.

\begin{theorem}
For any $s>0$, let $H_{s}=(V_{s},(\cdot,\cdot)_{s})$ denote the $s^{th}$
left-definite space associated with $(H,T)$. Then there exists a unique
left-definite operator $T_{s}$ in $H_{s}$ associated with $(H,T)$.
Furthermore, $T_{s}x=Tx$ for all $s>0$ and $x\in\mathcal{D}(T_{s}),$ where
\[
\mathcal{D}(T_{s})=V_{s+2}\subset\mathcal{D}(T).
\]

\end{theorem}

The last theorem that we list in this section shows that there is a
non-trivial left-definite theory only in the unbounded case; see \cite[Section
8]{Littlejohn-Wellman}.

\begin{theorem}
\label{Bounded-Unbounded} Let $H=(V,(\cdot,\cdot))$ be a Hilbert space.
Suppose $T:\mathcal{D}(T)\subset H\rightarrow H$ is a self-adjoint operator
that is bounded below by $rI$ for some $r>0.$ For each $s>0,$ let
$H_{s}=(V_{s},(\cdot,\cdot)_{s})$ and $A_{s}$ denote the $s^{th}$
left-definite space and $s^{th}$ left-definite operator, respectively,
associated with $(H,T).$
\end{theorem}

\begin{enumerate}
\item[(1)] \textit{Suppose }$T$\textit{ is bounded. Then, for each} $s>0,$

\begin{enumerate}
\item[(i)] $V=V_{s};$

\item[(ii)] \textit{the inner products} $(\cdot,\cdot)$ and $(\cdot,\cdot
)_{s}$ \textit{are equivalent};

\item[(iii)] $T=T_{s}.$
\end{enumerate}

\item[(2)] \textit{Suppose }$T$\textit{ is unbounded}. \textit{Then}

\begin{enumerate}
\item[(i)] $V_{s}$\textit{ is a proper subspace of }$V;$

\item[(ii)] $V_{s}$\textit{ is a proper subspace of }$V_{t}$\textit{ whenever
}$0<t<s;$

\item[(iii)] \textit{the inner products }$(\cdot,\cdot)$\textit{ and }%
$(\cdot,\cdot)_{s}$\textit{ are not equivalent for any }$s>0;$

\item[(iv)] \textit{the inner products }$(\cdot,\cdot)_{s}$\textit{ and
}$(\cdot,\cdot)_{t}$\textit{ are not equivalent for any }$s,t>0,$ $s\neq t;$

\item[(v)] $\mathcal{D}(T_{s})$\textit{ is a proper subspace of }%
$\mathcal{D}(T)$\textit{ for each }$s>0;$

\item[(vi)] $\mathcal{D}(T_{s})$\textit{ is a proper subspace of }%
$\mathcal{D}(T_{t})$\textit{ whenever }$0<t<s;$

\item[(vii)] \textit{the point spectrum of each }$T_{s}$ and $T$ are equal;
that is, $\sigma_{p}(T_{s})=\sigma_{p}(T)$ for each $s>0;$

\item[(viii)] \textit{the continuous spectrum of each }$T_{s}$ and $T$ are
equal; that is, $\sigma_{c}(T_{s})=\sigma_{c}(T)$ for each $s>0.$
\end{enumerate}
\end{enumerate}

From (\ref{Jacobi operator bounded below}), we see that there is a non-trivial
left-definite theory associated with the Jacobi self-adjoint operator $A$
defined in (\ref{Jacobi operator}). It is natural to ask: what are the
left-definite spaces $\{H_{s}\}_{s>0}$ and left-definite operators
$\{T_{s}\}_{s>0}$ associated with $(L_{-1,-1}^{2}(-1,1),A)?$ To answer this,
we observe from Definition \ref{left-definite space definition}(v) that the
$s^{th}$ left-definite inner product $(\cdot,\cdot)_{s}$ is generated by the
$s^{th}$ power of $A,$ or in our case, the $s^{th}$ composite power of the
Lagrangian symmetric form of $\ell\lbrack\cdot]$. Practically speaking, we can
only effectively determine these powers when $s$ is a positive integer.

\section{Left-Definite Theory of the Jacobi Differential Expression when
$\alpha=\beta=-1$\label{Left-definite Jacobi}}

In \cite{EKLWY}, the authors show that, for each $n\in\mathbb{N},$ the
$n^{th}$ composite power of the Jacobi differential expression
(\ref{Jacobi DE}) is given by%
\[
\ell^{n}[y](x)=\sum_{j=0}^{n}(-1)^{j}c_{j}(n,k)\left(  (1-x^{2})^{j}%
y^{(j)}(x)\right)  ^{(j)},
\]
where the coefficients $c_{j}(n,k)$ $(j=0,1,\ldots n)$ are non-negative and
defined as
\begin{equation}
c_{0}(n,k):=\left\{
\begin{array}
[c]{ll}%
0 & \text{if }k=0\\
k^{n} & \text{if }k>0
\end{array}
\right.  \text{ and }c_{j}(n,k):=\left\{
\begin{array}
[c]{ll}%
\genfrac{\{}{\}}{0pt}{}{n}{j}%
_{0} & \text{if }k=0\\
\sum_{r=0}^{n-j}\binom{n}{r}%
\genfrac{\{}{\}}{0pt}{}{n-r}{j}%
_{0}k^{r} & \text{if }k>0.
\end{array}
\right.  \label{c_j(n,k)}%
\end{equation}
In (\ref{c_j(n,k)}), the numbers $%
\genfrac{\{}{\}}{0pt}{}{n}{j}%
_{0}$ are called \textit{Jacobi-Stirling numbers}, defined by%
\[%
\genfrac{\{}{\}}{0pt}{}{n}{j}%
_{0}:=\delta_{n,j}\quad(j=0,1;\text{ }n\in\mathbb{N}_{0}),
\]
and, when $j\geq2$ and $n\in\mathbb{N}_{0},$ by%
\begin{equation}%
\genfrac{\{}{\}}{0pt}{}{n}{j}%
_{0}:=\sum_{r=2}^{j}(-1)^{r+j}\frac{(2r-1)(r-2)!\left[  r(r-1)\right]  ^{n}%
}{r!(j-r)!(j+r-1)!}. \label{Jacobi-Stirling New Notation}%
\end{equation}
The coefficients $c_{j}(n,k)$ were originally defined in \cite{EKLWY} through
the identity
\begin{equation}
\sum\limits_{j=0}^{n}c_{j}(n,k)\frac{m!(m+j-2)!}{(m-j)!(m-2)!}=(m(m-1)+k)^{n}.
\label{Defining Identity}%
\end{equation}
For a discussion of the combinatorial properties of $%
\genfrac{\{}{\}}{0pt}{}{n}{j}%
_{0},$ and the more general Jacobi-Stirling numbers $%
\genfrac{\{}{\}}{0pt}{}{n}{j}%
_{\gamma}$, see the recent papers \cite{Andrews-Gawronski-Littlejohn},
\cite{Andrews-Egge-Gawronski-Littlejohn}, and \cite{Gelineau-Zeng}. The
following table lists some of these Jacobi-Stirling numbers $%
\genfrac{\{}{\}}{0pt}{}{n}{j}%
_{0}$ for small values of $n$ and $j$.%

\[%
\begin{tabular}
[c]{|l|l|l|l|l|l|l|l|l|l|}\hline
$j/n$ & $n=0$ & $n=1$ & $n=2$ & $n=3$ & $n=4$ & $n=5$ & $n=6$ & $n=7$ &
$n=8$\\\hline
$j=0$ & $1$ & $0$ & $0$ & $0$ & $0$ & $0$ & $0$ & $0$ & $0$\\\hline
$j=1$ & $0$ & $1$ & $0$ & $0$ & $0$ & $0$ & $0$ & $0$ & $0$\\\hline
$j=2$ & $0$ & $0$ & $1$ & $2$ & $4$ & $8$ & $16$ & $32$ & $64$\\\hline
$j=3$ & $0$ & $0$ & $0$ & $1$ & $8$ & $52$ & $320$ & $1936$ & $11648$\\\hline
$j=4$ & $0$ & $0$ & $0$ & $0$ & $1$ & $20$ & $292$ & $3824$ & $47824$\\\hline
$j=5$ & $0$ & $0$ & $0$ & $0$ & $0$ & $1$ & $40$ & $1092$ & $25664$\\\hline
$j=6$ & $0$ & $0$ & $0$ & $0$ & $0$ & $0$ & $1$ & $70$ & $3192$\\\hline
$j=7$ & $0$ & $0$ & $0$ & $0$ & $0$ & $0$ & $0$ & $1$ & $112$\\\hline
$j=8$ & $0$ & $0$ & $0$ & $0$ & $0$ & $0$ & $0$ & $0$ & $1$\\\hline
\end{tabular}
\ \ \ \ \ \
\]

\begin{center}
\textbf{Jacobi-Stirling numbers }(e.g. $%
\genfrac{\{}{\}}{0pt}{}{7}{5}%
_{0}=1092$)
\end{center}

For example, if $k=0,$ we see from this table that%
\begin{align*}
\ell^{5}[y](x)  &  =-\left(  \mathbf{1}(1-x^{2})^{5}y^{(5)}(x)\right)
^{(5)}+\left(  \mathbf{20}(1-x^{2})^{4}y^{(4)}(x)\right)  ^{(4)}\\
&  -\left(  \mathbf{52}(1-x^{2})^{3}y^{\prime\prime\prime}(x)\right)
^{\prime\prime\prime}+\left(  \mathbf{8}(1-x^{2})^{2}y^{\prime\prime
}(x)\right)  ^{\prime\prime}.
\end{align*}
It is interesting to note that $%
\genfrac{\{}{\}}{0pt}{}{n}{j}%
_{0}=%
\genfrac{\{}{\}}{0pt}{}{n-1}{j-1}%
_{1};$ the numbers $%
\genfrac{\{}{\}}{0pt}{}{n}{j}%
_{1}$ are called the \textit{Legendre-Stirling numbers} which are the subject
of several recent papers (see, for example, \cite{Andrews-Littlejohn},
\cite{Andrews-Gawronski-Littlejohn}, \cite{Andrews-Egge-Gawronski-Littlejohn},
and \cite{Egge}).

Let $k>0.$ For each $n\in%
\mathbb{N}
,$ define the inner product space
\begin{equation}
H_{n}:=\left(  V_{n};(\cdot,\cdot)_{n}\right)  , \label{H_n}%
\end{equation}
where%
\begin{equation}
V_{n}:=\{f:(-1,1)\rightarrow%
\mathbb{C}
\mid f\in AC_{\mathrm{loc}}^{(n-1)}(-1,1);f^{(j)}\in L^{2}\left(
(-1,1);(1-x^{2}\right)  ^{j-1}),j=0,1,..,n\}
\label{left-definite space definition -1,-1}%
\end{equation}
and%
\begin{equation}
(f,g)_{n}:=%
{\displaystyle\sum\limits_{j=0}^{n}}
c_{j}(n,k)\int_{-1}^{1}f^{(j)}(x)\overline{g}^{(j)}(x)(1-x^{2})^{j-1}dx.
\label{LD IP}%
\end{equation}
For later purposes, we note that
\begin{equation}
(f,g)_{1}=\int_{-1}^{1}(f^{\prime}(x)\overline{g}^{\prime}(x)+kf(x)\overline
{g}(x)(1-x^{2})^{-1})dx. \label{LD IP 1}%
\end{equation}

We shall show that $H_{n}$ is the $n^{th}$ left-definite space associated with
the pair $\left(  L_{-1,-1}^{2}(-1,1);A\right)  ,$ where $A$ is the
self-adjoint Jacobi operator defined in (\ref{Jacobi operator}). Mimicking the
results from \cite{EKLWY}\ \textit{mutatis mutandis}, Theorem
\ref{(2) Hermite}\ follows; for a specific proof see the thesis \cite{Bruder}
of Bruder.

\begin{theorem}
\label{(2) Hermite}Let $k>0$. For each $n\in%
\mathbb{N}
,$ $H_{n}$ is a Hilbert space.
\end{theorem}

It is clear, from (\ref{c_j(n,k)}) and the non-negativity of each
$c_{j}(n,k),$ that
\begin{equation}
\left(  f,f\right)  _{n}=\sum\limits_{j=0}^{n}c_{j}(n,k)\left\Vert
f^{(j)}\right\Vert _{j-1,j-1}^{2}\geq c_{0}(n,k)\left\Vert f\right\Vert
_{-1,-1}^{2}=k^{n}\left(  f,f\right)  _{-1,-1}\quad\left(  f\in H_{n}\right)
. \label{LD Inequality}%
\end{equation}
\ 

Since the proofs of the next two results are similar to, respectively, the
proofs given in Theorem 7 and Lemma 8 in \cite{BruderLittlejohn}, we omit them.

\begin{theorem}
\label{(3) Hermite}The Jacobi polynomials $\left\{  P_{m}^{(-1,-1)}\right\}
_{m=2}^{\infty}$ form a complete orthogonal set in $H_{n}$ for each
$n\in\mathbb{N}.$ Equivalently, the vector space $\mathcal{P}_{-1}[-1,1],$
defined in Lemma \ref{completeness -1-1}, is dense in $H_{n}.$
\end{theorem}

\begin{lemma}
\label{Colorado}For $n\geq2$ and $p,q\in\mathcal{P}_{-1}[-1,1],$%
\[
(p,q)_{n}=\left(  A^{n}p,q\right)  _{-1,-1}.
\]

\end{lemma}

We are now in position to prove the main result of this section.

\begin{theorem}
\label{Left-Definite Jacobi Theorem}For $k>0,$ let%
\[
A:\mathcal{D}\left(  A\right)  \subset L^{2}\left(  (-1,1);(1-x^{2}%
)^{-1}\right)  \rightarrow L^{2}\left(  (-1,1);(1-x^{2})^{-1}\right)
\]
be the Jacobi self-adjoint operator defined in $($\ref{Jacobi operator}$)$
having the Jacobi polynomials $\left\{  P_{m}^{(-1,-1)}\right\}
_{m=2}^{\infty}$ as eigenfunctions. For each $n\in%
\mathbb{N}
$, let $H_{n},$ $V_{n},$ and $(\cdot,\cdot)_{n}$ be as given in $($%
\ref{H_n}$),$ $($\ref{left-definite space definition -1,-1}$),$ and
$($\ref{LD IP}$),$ respectively. Then $H_{n}$ is the $n^{th}$-left-definite
space associated with $\left(  L^{2}\left(  (-1,1);(1-x^{2}\right)
^{-1}),A\right)  $. Moreover, the Jacobi polynomials $\left\{  P_{m}%
^{(-1,-1)}\right\}  _{m=2}^{\infty}$ form a complete orthogonal set in each
$H_{n}$; specifically, they satisfy the orthogonality relation%
\[
\left(  P_{m}^{(-1,-1)},P_{l}^{(-1,-1)}\right)  _{n}=(m(m-1)+k)^{n}%
\delta_{m,l.}%
\]
Furthermore, define%
\[
B_{n}:=\mathcal{D}\left(  B_{n}\right)  \subset H_{n}\rightarrow H_{n}%
\]
by%
\[
B_{n}f:=\ell\left[  f\right]  \text{ \ \ \ \ }\left(  f\in\mathcal{D}\left(
B_{n}\right)  :=V_{n+2}\right)  .
\]
Then $B_{n}$\ is the $n^{th}$ left-definite $($self-adjoint$)$ operator
associated with the pair \noindent$\left(  L_{-1,-1}^{2}(-1,1),A\right)  $.
Lastly, the spectrum of $B_{n}$\ is given by%
\[
\sigma\left(  B_{n}\right)  =\left\{  m(m-1)+k\mid m=2,3,4,\ldots\right\}
=\sigma(A).
\]

\end{theorem}

\begin{proof}
Fix $n\in\mathbb{N}$. We need to show that $H_{n}$ satisfies the five
properties given in Definition \ref{left-definite space definition}.\smallskip

\noindent\textbf{(i)} $H_{n}$ is a Hilbert space; this is the statement given
in Theorem \ref{(2) Hermite}.\medskip\ \newline\noindent\textbf{(ii)} We need
to show that $\mathcal{D}\left(  A^{n}\right)  \subset V_{n}.$ Let
$f\in\mathcal{D}\left(  A^{n}\right)  $. Since the Jacobi polynomials
$\left\{  P_{m}^{(-1,-1)}\right\}  _{m=2}^{\infty}$ form a complete
orthonormal set in $L_{-1,-1}^{2}(-1,1)$, we see that%
\begin{equation}
p_{j}\rightarrow f\text{ \ \ \ \ in }L_{-1,-1}^{2}(-1,1)\text{ as
}j\rightarrow\infty, \label{star Beethoven}%
\end{equation}
where%
\[
p_{j}(t)=\sum\limits_{m=2}^{j}c_{m}P_{m}^{(-1,-1)}(t)\text{ \ \ \ \ }\left(
t\in(-1,1)\right)  ,
\]
and%
\[
c_{m}:=\left(  f,P_{m}^{(-1,-1)}\right)  _{-1,-1}=\int_{-1}^{1}f(t)P_{m}%
^{(-1,-1)}(t)\left(  1-t^{2}\right)  ^{-1}dt\text{ \ \ \ \ }(m\geq2).
\]
Since $A^{n}f\in L_{-1,-1}^{2}(-1,1)$, we see that, as $j\rightarrow\infty,$%
\[
\sum\limits_{m=2}^{j}\widetilde{c}_{m}P_{m}^{(-1,-1)}\rightarrow A^{n}f\text{
\ \ \ \ in }L_{-1,-1}^{2}(-1,1)
\]
However, from the self-adjointness of $A^{n},$ we find that%
\begin{align*}
\widetilde{c}_{m}  &  :=\left(  A^{n}f,P_{m}^{(-1,-1)}\right)  _{-1,-1}%
=\left(  f,A^{n}P_{m}^{(-1,-1)}\right)  _{-1,-1}\\
&  =(m(m-1)+k)^{n}\left(  f,P_{m}^{(-1,-1)}\right)  _{-1,-1}\\
&  =(m(m-1)+k)^{n}c_{m};
\end{align*}
Consequently,%
\[
A^{n}p_{j}\rightarrow A^{n}f\text{ \ \ \ \ in }L_{-1,-1}^{2}(-1,1)\text{ as
}j\rightarrow\infty.
\]
Moreover, by Lemma \ref{Colorado},%
\begin{align*}
\left(  \left\Vert p_{j}-p_{r}\right\Vert _{n}\right)  ^{2}  &  =\left(
A^{n}\left[  p_{j}-p_{r}\right]  ,p_{j}-p_{r}\right)  _{-1,-1}\\
&  \rightarrow0\text{ \ \ \ \ as }j,r\rightarrow\infty
\end{align*}
i.e. $\left\{  p_{j}\right\}  _{j=0}^{\infty}$ is Cauchy in $H_{n}$. Since
$H_{n}$ is a Hilbert space, there exists $g\in H_{n}\subset L_{-1,-1}%
^{2}(-1,1)$ such that%
\[
p_{j}\rightarrow g\text{ \ \ \ \ in }H_{n}\text{ as }j\rightarrow\infty.
\]
Furthermore, from (\ref{LD Inequality}), we see that
\[
\left\Vert p_{j}-g\right\Vert _{-1,-1}\leq k^{-n/2}\left\Vert p_{j}%
-g\right\Vert _{n},
\]
and hence,%
\begin{equation}
p_{j}\rightarrow g\text{ \ \ \ \ in }L_{-1,-1}^{2}(-1,1).
\label{double star Beethoven}%
\end{equation}
Comparing (\ref{star Beethoven}) and (\ref{double star Beethoven}),
\[
f=g\in H_{n}.
\]
\medskip\newline\noindent\textbf{(iii)} By Theorem \ref{(3) Hermite},
$\mathcal{P}_{-1}[-1,1]$ is dense in $H_{n}.$ Since $\mathcal{P}_{-1}[-1,1]$
$\subset\mathcal{D}\left(  A^{n}\right)  ,$ we see that $\mathcal{D}\left(
A^{n}\right)  $ is dense in $H_{n}$.\medskip\ \newline\noindent\textbf{(iv)}
We already showed, in (\ref{LD Inequality}), that $\left(  f,f\right)
_{n}\geq k^{n}\left(  f,f\right)  _{-1,-1}$ for all $f\in V_{n}.\medskip$
\newline\noindent\textbf{(v) }We need to show that $\left(  f,g\right)
_{n}=\left(  A^{n}f,g\right)  _{-1,-1}$ for $f\in\mathcal{D}\left(
A^{n}\right)  $ and $g\in V_{n}.$ This is true for any $f,g\in\mathcal{P}%
_{-1}[-1,1]$ by Lemma \ref{Colorado}. Let $f\in\mathcal{D}\left(
A^{n}\right)  \subset H_{n}$, $g\in H_{n}.$ Since $\mathcal{P}_{-1}[-1,1]$ is
dense in both $H_{n}$ and $L_{-1,-1}^{2}(-1,1)$ (see Lemma
\ref{completeness -1-1} and Theorem \ref{(3) Hermite}), and (by (\textbf{iv}%
)), convergence in $H_{n}$ implies convergence in $L_{-1,-1}^{2}(-1,1)$, there
exist sequences $\left\{  p_{j}\right\}  _{j=0}^{\infty},\left\{
q_{j}\right\}  _{j=0}^{\infty}$ $\subset\mathcal{P}_{-1}[-1,1]$ such that
\begin{align*}
p_{j}  &  \rightarrow f\text{ \ \ \ \ in }H_{n}\text{ as }j\rightarrow\infty\\
A^{n}p_{j}  &  \rightarrow A^{n}f\text{ \ \ in }L_{-1,-1}^{2}(-1,1)\text{ as
}j\rightarrow\infty\quad(\text{from the proof of part }(ii))
\end{align*}
and%
\[
q_{j}\rightarrow g\text{ \ \ \ \ in }H_{n}\text{ and }L_{-1,-1}^{2}%
(-1,1)\text{ as }j\rightarrow\infty.
\]
Hence,
\begin{align*}
\left(  A^{n}f,g\right)  _{-1,-1}  &  =\lim_{j\rightarrow\infty}\left(
A^{n}p_{j},q_{j}\right)  _{-1,-1}\\
&  =\lim_{j\rightarrow\infty}\left(  p_{j},q_{j}\right)  _{n}\text{ by Lemma
\ref{Colorado}}\\
&  =\left(  f,f\right)  _{n}.
\end{align*}
The results listed in the theorem on the left-definite operator $B_{n}$ and
the spectrum of $B_{n}$ follow immediately from the general left-definite
theory discussed in Section \ref{Left-definite theory}.
\end{proof}

\section{Sobolev Orthogonality and Spectral Theory of the Jacobi
Expression\label{Sobolev Spectral Theory}}

As discussed in Section \ref{Introduction}, any polynomial $p(x)$ of degree
one is a solution of the Jacobi differential equation%
\[
\ell\lbrack y](x):=-(1-x^{2})y^{\prime\prime}(x)+ky(x)=(n(n-1)+k)y(x);
\]
moreover, it is important to note that the Jacobi polynomial $P_{1}%
^{(-1,-1)}(x)$ as defined, say, in \cite{Szego}, is identically zero. If we define%

\begin{equation}
\widetilde{P}_{0}^{(-1,-1)}(x):=1,\quad\widetilde{P}_{1}^{(-1,-1)}%
(x):=x/\sqrt{3}\label{Jacobi polynomials redefined -1,-1}%
\end{equation}
and renormalize the Jacobi polynomials (\ref{Jacobi polynomial definition a,b}%
), for $\alpha=\beta=-1,$ of degree $n\geq2,$ by%
\begin{equation}
\widetilde{P}_{n}^{(-1,-1)}(x):=\frac{\sqrt{4n-2}}{\left(  n-1\right)  }%
{\displaystyle\sum\limits_{j=0}^{n}}
\binom{n-1}{n-j}\binom{n-1}{j}\left(  \frac{x-1}{2}\right)  ^{j}\left(
\frac{x+1}{2}\right)  ^{n-j},\label{Renormalized Jacobi polynomials}%
\end{equation}
then Kwon and Littlejohn prove the following theorem in \cite{Kwon-Littlejohn}.

\begin{theorem}
\label{orthogonal Sobolev}The Jacobi polynomials $\left\{  \widetilde{P}%
_{n}^{(-1,-1)}(x)\right\}  _{n=0}^{\infty}$ , as given in $($%
\ref{Jacobi polynomials redefined -1,-1}$)$ and $($%
\ref{Renormalized Jacobi polynomials}$)$, are orthonormal with respect to the
Sobolev inner product%
\begin{equation}
\phi\left(  f,g\right)  :=\frac{1}{2}f(-1)\overline{g}(-1)+\frac{1}%
{2}f(1)\overline{g}(1)+\int_{-1}^{1}f^{\prime}(x)\overline{g}^{\prime}(x)dx.
\label{phi inner product}%
\end{equation}

\end{theorem}

A key step in establishing this orthogonality is the fact that $\widetilde{P}%
_{n}^{(-1,-1)}(\pm1)=0$ for $n\geq2$; see (\ref{Value at -1 and +1}).

\begin{definition}
Define%
\[
W:=\left\{  f:\left[  -1,1\right]  \rightarrow%
\mathbb{C}
\mid f\in AC\left[  -1,1\right]  ;f^{\prime}\in L^{2}(-1,1)\right\}
\]
and, with $\phi(\cdot,\cdot)$ being the inner product defined in
$($\ref{phi inner product}$),$ let $\left\Vert f\right\Vert _{\phi}%
:=\phi(f,f)^{1/2}$ $(f\in W)$ be the associated norm.
\end{definition}

\begin{theorem}
$\left(  W,\phi(\cdot,\cdot)\right)  $ is a Hilbert space.
\end{theorem}

\begin{proof}
Let $\left\{  f_{n}\right\}  \subset W$ be a Cauchy sequence. Hence%
\begin{align*}
\left\Vert f_{n}-f_{m}\right\Vert _{\phi}^{2}  &  =\frac{1}{2}\left\vert
f_{n}(-1)-f_{m}(-1)\right\vert ^{2}+\frac{1}{2}\left\vert f_{n}(1)-f_{m}%
(1)\right\vert ^{2}+\int_{-1}^{1}\left\vert f_{n}^{\prime}(x)-f_{m}^{\prime
}(x)\right\vert ^{2}dx\\
&  \rightarrow0\text{ \ \ \ \ as }n,m\rightarrow\infty.
\end{align*}
In particular, since%
\[
\int_{-1}^{1}\left\vert f_{n}^{\prime}(x)-f_{m}^{\prime}(x)\right\vert
^{2}dx\leq\left\Vert f_{n}-f_{m}\right\Vert _{\phi}^{2},
\]
we see that $\left\{  f_{n}^{\prime}\right\}  $ is Cauchy in $L^{2}(-1,1)$.
Since $L^{2}(-1,1)$ is complete, there exists $g\in L^{2}(-1,1)$ such that%
\begin{equation}
f_{n}^{\prime}\rightarrow g\text{ \ \ \ \ as }n\rightarrow\infty\text{ \ in
}L^{2}(-1,1). \label{(1) Mendelssohn}%
\end{equation}
Also, since%
\[
\frac{1}{2}\left\vert f_{n}(-1)-f_{m}(-1)\right\vert ^{2}\leq\left\Vert
f_{n}-f_{m}\right\Vert _{\phi}^{2}\text{\quad and\quad}\frac{1}{2}\left\vert
f_{n}(1)-f_{m}(1)\right\vert ^{2}\leq\left\Vert f_{n}-f_{m}\right\Vert _{\phi
}^{2},
\]
we see that the sequences $\left\{  f_{n}(\pm1)\right\}  $ are both Cauchy in
$%
\mathbb{C}
$ and, hence, there exists $A_{\pm1}\in%
\mathbb{C}
$ such that%
\begin{align}
f_{n}(1)  &  \rightarrow A_{1}\text{ in }\mathbb{C}\label{(2) Mendelssohn}\\
f_{n}(-1)  &  \rightarrow A_{-1}\text{ in }\mathbb{C}. \label{(3) Mendelssohn}%
\end{align}
Furthermore, since $f_{n}\in AC\left[  -1,1\right]  $ $(n\in%
\mathbb{N}
)$, we see that%
\[
\int_{-1}^{1}g(t)dt\longleftarrow\int_{-1}^{1}f_{n}^{\prime}(t)dt=f_{n}%
(1)-f_{n}(-1)\rightarrow A_{1}-A_{-1};
\]
that is,
\begin{equation}
A_{1}=A_{-1}+\int_{-1}^{1}g(t)dt. \label{(4) Mendelssohn}%
\end{equation}
Define $f:\left[  -1,1\right]  \rightarrow%
\mathbb{C}
$ by%
\[
f(x)=A_{-1}+\int_{-1}^{x}g(t)dt.
\]
It is clear that $f\in AC\left[  -1,1\right]  $ and $f^{\prime}(x)=g(x)$ $\in
L^{2}(-1,1)$ for \ a.e. $x\in\left[  -1,1\right]  $, so $f\in W.$ Furthermore,
$f(-1)=A_{-1}$ and $f(1)=A_{-1}+\int_{-1}^{1}g(t)dt=A_{1}$ by
(\ref{(4) Mendelssohn}). Now%
\begin{align*}
\left\Vert f_{n}-f\right\Vert _{\phi}^{2}  &  =\frac{1}{2}\left\vert
f_{n}(-1)-f(-1)\right\vert ^{2}+\frac{1}{2}\left\vert f_{n}(1)-f(1)\right\vert
^{2}+\int_{-1}^{1}\left\vert f_{n}^{\prime}(t)-f^{\prime}(t)\right\vert
^{2}dt\\
&  =\frac{1}{2}\left\vert f_{n}(-1)-A_{-1}\right\vert ^{2}+\frac{1}%
{2}\left\vert f_{n}(1)-A_{1}\right\vert ^{2}+\int_{-1}^{1}\left\vert
f_{n}^{\prime}(t)-g(t)\right\vert ^{2}dt\\
&  \rightarrow0\text{ as }n\rightarrow\infty.
\end{align*}
Thus, $\left(  W,\phi(\cdot,\cdot)\right)  $ is complete.
\end{proof}

With $W$ and $\phi(\cdot,\cdot)$ as given above, define%
\begin{align*}
W_{1}  &  :=\left\{  f\in W\mid f(\pm1)=0\right\} \\
W_{2}  &  :=\left\{  f\in W\mid f^{\prime}(x)=c\text{ for some constant
}c=c(f)\right\}  .
\end{align*}

\begin{remark}
\label{Jacobi and W_1,2}It is clear, from the definition, that $W_{2}$ is
two-dimensional and, in fact, $W_{2}=\mathrm{span}\{\widetilde{P}%
_{0}^{(-1,-1)},$ $\widetilde{P}_{1}^{(-1,-1)}\}$.
\end{remark}

\begin{theorem}
\label{Fundamental Decomposition}The spaces $W_{1}$ and $W_{2}$ are closed,
orthogonal subspaces of $\left(  W,\phi\left(  \cdot,\cdot\right)  \right)  $
and%
\[
W=W_{1}\oplus W_{2}.
\]

\end{theorem}

\begin{proof}
Since $W_{2}$ is two-dimensional, it is a closed subspace of $W$. By
definition, the orthogonal complement of $W_{2}$ is given by%
\[
W_{2}^{\perp}=\left\{  f\in W\mid\phi(f,g)=0\text{ \ for all }g\in
W_{2}\right\}  .
\]
To see that $W_{1}\subset W_{2}^{\perp}$, let $f\in W_{1}$, $g\in W_{2}$ and
consider%
\[
\phi(f,g)=\frac{1}{2}f(-1)\overline{g}(-1)+\frac{1}{2}f(1)\overline{g}%
(1)+\int_{-1}^{1}f^{\prime}(x)\overline{g}^{\prime}(x)dx.
\]
The first two terms on the right hand side vanish since $f\in W_{1}$;
furthermore, $\overline{g}^{\prime}(x)=c$ for some constant $c\in\mathbb{C}$
since $g\in W_{2}$. Moreover,%
\begin{align*}
\phi\left(  f,g\right)   &  =\int_{-1}^{1}f^{\prime}(x)\overline{g}^{\prime
}(x)dx=c\int_{-1}^{1}f^{\prime}(x)dx\\
&  =c\left(  f(1)-f(-1)\right)  =0,
\end{align*}
so $f\in W_{2}^{\perp}.$ Conversely, let $f\in W_{2}^{\bot}.$ Then, for any
choice of constants $A,B\in\mathbb{C},$ it is the case that%
\begin{align*}
0 &  =\phi\left(  f(x),Ax+B\right)  \\
&  =\dfrac{1}{2}f(-1)(-\overline{A}+\overline{B})+\dfrac{1}{2}f(1)(\overline
{A}+\overline{B})+\overline{A}\int_{-1}^{1}f^{\prime}(x)dx\\
&  =-\frac{3}{2}\overline{A}f(-1)+\frac{3}{2}\overline{A}f(1)+\frac
{\overline{B}}{2}(f(-1)+f(1)).
\end{align*}
By choosing $A=0,$ $B\neq0$ and then $A\neq0$ and $B=0,$ we find that
$f(\pm1)=0$ so $f\in W_{1}.$
\end{proof}

We note that, given $f\in W,$ we can (uniquely) write%
\[
f=f_{1}+f_{2}\quad(f_{i}\in W_{i}\text{ }(i=1,2)),
\]
where
\begin{equation}
f_{1}(x):=f(x)-f_{2}(x)\text{ and }f_{2}(x):=\frac{f(1)-f(-1)}{2}%
x+\frac{f(1)+f(-1)}{2}\quad(x\in\lbrack-1,1]). \label{Representation1}%
\end{equation}

We now turn our attention to the construction of the self-adjoint operator $T$
in $(W,\phi(\cdot,\cdot)),$ generated by the Jacobi differential expression
$\ell\lbrack\cdot]$ given in (\ref{Jacobi DE}), that has the \textit{entire}
sequence of Jacobi polynomials $\left\{  \widetilde{P}_{n}^{(-1,-1)}%
(x)\right\}  _{n=0}^{\infty}$ as eigenfunctions. The main idea is to use the
decomposition in Theorem \ref{Fundamental Decomposition} to construct
self-adjoint operators $T_{1}$ in $W_{1}$ and $T_{2}$ in $W_{2},$ both
generated by $\ell\lbrack\cdot],$ that have, respectively, the Jacobi
polynomials $\left\{  \widetilde{P}_{n}^{(-1,-1)}(x)\right\}  _{n=2}^{\infty}$
and $\left\{  \widetilde{P}_{0}^{(-1,-1)}(x),\widetilde{P}_{1}^{(-1,-1)}%
(x)\right\}  $ as eigenfunctions. The operator $T$ is then specifically
defined to be the direct sum of $T_{1}$ and $T_{2}.$ The construction of
$T_{2}$ is straightforward, but constructing $T_{1}$ needs special attention.
Indeed, it requires the first left-definite operator $B_{1},$ defined in
Theorem \ref{Left-Definite Jacobi Theorem}, associated with the pair
$(A,L^{2}((-1,1);(1-x^{2})^{-1}),$ where $A$ is the self-adjoint operator
defined in (\ref{Jacobi operator}). The construction of $T_{1}$ begins with
the following remarkable, and surprising, identification of the function
spaces $W_{1}$ and $V_{1}.$

\begin{theorem}
\label{W1,1=V1}$W_{1}=V_{1},$ where $V_{1}$ is defined as in $($%
\ref{left-definite space definition -1,-1}$)$.
\end{theorem}

\begin{proof}
For the sake of completeness, we note that%
\[
V_{1}=\{f:(-1,1)\rightarrow\mathbb{C}\mid f\in AC_{\mathrm{loc}}%
(-1,1);(1-x^{2})^{-1/2}f,f^{\prime}\in L^{2}(-1,1)\}.
\]
and observe that the condition $(1-x^{2})^{-1/2}f\in L^{2}(-1,1)$ is
equivalent to $f\in L_{-1,-1}^{2}(-1,1).\medskip$

\noindent\textbf{(1)} We first show that $V_{1}\subseteq W_{1}$. Let $f\in
V_{1}.$ In particular, $f\in AC[-1,1].$ For $0\leq x<1$,
\[
\int_{0}^{x}f^{\prime}(t)dt=f(x)-f(0);
\]
consequently, since $f^{\prime}\in L^{2}(-1,1)\subset L^{1}(-1,1),$ we see
that $\lim\limits_{x\rightarrow1^{-}}f(x)$ exists and is finite. Similarly,
$\lim\limits_{x\rightarrow-1^{+}}f(x)$ exists and is finite. Define%
\[
f(\pm1):=\lim\limits_{x\rightarrow\pm1^{\mp}}f(x),
\]
so $f\in AC\left[  -1,1\right]  .$ It suffices to show that $f(\pm1)=0.$
Suppose that $f(1)\neq0.$ Hence, for some $c>0,$ there exists $0<\delta<1$
such that%
\[
\left\vert f(x)\right\vert \geq c>0
\]
for all $x\in\left[  \delta,1\right]  .$ Since $f\in L_{-1,-1}^{2}(-1,1)$, we
see that%
\begin{align*}
\infty &  >\int_{0}^{1}\left\vert f(x)\right\vert ^{2}(1-x^{2})^{-1}dx\\
&  \geq\int_{0}^{\delta}\left\vert f(x)\right\vert ^{2}(1-x^{2})^{-1}dx\geq
c^{2}\int_{0}^{\delta}(1-x^{2})^{-1}dx=\infty,
\end{align*}
a contradiction. Hence, $f(1)=0$; similarly, $f(-1)=0$, so $f\in
W_{1}.\medskip$

\noindent\textbf{(2)} Let $f\in W_{1}.$ It suffices to show that $f\in
L^{2}\left(  (-1,1);(1-x^{2})^{-1}\right)  $. For $-1<x<0$,
\[
(1-x^{2})^{-1/2}\int_{-1}^{x}f^{\prime}(t)dt=(1-x^{2})^{-1/2}f(x)
\]
since $f(-1)=0$. We use Theorem \ref{CHEL} on $(-1,0]$ with%
\[
\psi(x)=(1-x^{2})^{-1/2},\text{ }\varphi(x)=1.
\]
Clearly, $\psi$ is square integrable near $0$ and $\varphi$ is square
integrable near $-1.$ Moreover,%
\[
\int_{-1}^{x}dt\int_{x}^{0}\frac{dt}{1-t^{2}}\leq\int_{-1}^{x}dt\int_{x}%
^{0}\frac{dt}{1+t}=-(x+1)\ln(1+x),
\]
which is a bounded function on $(-1,0].$ By Theorem \ref{CHEL}, it follows
that
\[
f\in L^{2}\left(  (-1,0];(1-x^{2})^{-1}\right)  ;
\]
a similar argument shows $f\in L^{2}\left(  [0,1);(1-x^{2})^{-1}\right)  $.
Hence $W_{1}\subset V_{1}$.
\end{proof}

\begin{theorem}
\label{Equivalent Inner Products}The inner products $\phi(\cdot,\cdot)$ and
$(\cdot,\cdot)_{1}$, where $(\cdot,\cdot)_{1}$ is defined in $($%
\ref{LD IP 1}$),$ are equivalent on $W_{1}=V_{1}.$
\end{theorem}

\begin{proof}
First of all, we note that both $\left(  W_{1},\phi(\cdot,\cdot)\right)  $ and
$\left(  V_{1},(\cdot,\cdot)_{1}\right)  $ are Hilbert spaces. Let $f\in
W_{1}=V_{1}.$ Since%
\[
\left\Vert f\right\Vert _{\phi}^{2}=\int_{-1}^{1}\left\vert f^{\prime
}(x)\right\vert ^{2}dx\leq\int_{-1}^{1}\left[  \left\vert f^{\prime
}(x)\right\vert ^{2}+k\left\vert f(x)\right\vert ^{2}\left(  1-x^{2}\right)
^{-1}\right]  dx=\left\Vert f\right\Vert _{1}^{2},
\]
we see, by the Open Mapping Theorem (see \cite[Theorem 4.12-2 and Problem 9,
p. 291]{Kreyszig}), that these inner products are equivalent.
\end{proof}

\begin{remark}
\label{Jacobi polynomials and W_1}Since, by Theorem \ref{(2) Hermite}, the
Jacobi polynomials $\left\{  \widetilde{P}_{n}^{(-1,-1)}\right\}
_{n=2}^{\infty}$ form a complete orthogonal set in the first left-definite
space $H_{1}=(V_{1},(\cdot,\cdot)_{1}),$ it follows from Theorem
\ref{Equivalent Inner Products} that they are also a complete orthogonal set
in $(W_{1},\phi(\cdot,\cdot)).$ Together with Remark \ref{Jacobi and W_1,2},
we see that the full sequence of Jacobi polynomials $\left\{  \widetilde{P}%
_{n}^{(-1,-1)}\right\}  _{n=0}^{\infty}$ form a complete orthogonal set in
$W=W_{1}\oplus W_{2}.$
\end{remark}

We now construct a self-adjoint operator $T_{1}$ in the space $W_{1}$,
generated by the Jacobi expression $\ell\lbrack\cdot],$ defined in
(\ref{Jacobi DE}), having the sequence of Jacobi polynomials $\{P_{n}%
^{(-1,-1)}\}_{n=2}^{\infty}$ as eigenfunctions. Recall that the first
left-definite operator
\[
B_{1}:\mathcal{D}\left(  B_{1}\right)  :=V_{3}\subset H_{1}\rightarrow H_{1},
\]
associated with $\left(  A,L_{-1,-1}^{2}(-1,1)\right)  $, is self-adjoint in
the first left-definite space $H_{1}$ (see (\ref{H_n})) and given specifically
by%
\[
B_{1}[f](x):=\ell\lbrack f](x)=-(1-x^{2})f^{\prime\prime}(x)+kf(x),
\]
where $f\in\mathcal{D}\left(  B_{1}\right)  :=V_{3}$
\begin{align*}
&  =\left\{  f:(-1,1)\rightarrow%
\mathbb{C}
\mid f,f^{\prime},f^{\prime\prime}\in AC_{\mathrm{loc}}(-1,1);\right.
\quad\quad\quad\\
&  \qquad\left.  (1-x^{2})f^{\prime\prime\prime},(1-x^{2})^{1/2}%
f^{\prime\prime},f^{\prime},(1-x^{2})^{-1/2}f\in L^{2}(-1,1)\right\}  .
\end{align*}
More specifically, $B_{1}$ is self-adjoint with respect to the first
left-definite inner product $(\cdot,\cdot)_{1}.$ We now set out to prove that
the operator $T_{1}:\mathcal{D}(T_{1})\subset W_{1}\rightarrow W_{1}$ given by%
\begin{align*}
T_{1}f &  =B_{1}f=\ell\lbrack f]\\
f &  \in\mathcal{D}(T_{1}):=V_{3}%
\end{align*}
is self-adjoint in $\left(  W_{1},\phi(\cdot,\cdot)\right)  .$

\begin{theorem}
\label{(4) Hermite}Let $f,g\in V_{3}.$ Then%
\[
\lim_{x\rightarrow\pm1^{\mp}}(1-x^{2})f^{\prime\prime}(x)\overline{g}^{\prime
}(x)=0.
\]

\end{theorem}

\noindent

\begin{proof}
It suffices to prove this result for $x\rightarrow1^{-}.$ Let $f,g\in V_{3}.$
Without loss of generality, assume that $f,g$ are both real-valued. Since
$V_{3}\subset V_{1}$ and $T_{1}f\in V_{1},$ we see that%
\[
f^{\prime},(T_{1}f)^{\prime},g^{\prime}\in L^{2}(-1,1).
\]
Hence $(T_{1}f)^{\prime}g^{\prime},f^{\prime}g^{\prime}\in L^{1}(-1,1).$ For
$0\leq x<1$,
\[
\int_{0}^{x}(T_{1}f)^{\prime}(t)g^{\prime}(t)dt=-\int_{0}^{x}\left(
(1-t^{2})f^{\prime\prime}(t)\right)  ^{\prime}g^{\prime}(t)dt+k\int_{0}%
^{x}f^{\prime}(t)g^{\prime}(t)dt.
\]
It follows that
\begin{equation}
\lim_{x\rightarrow1^{-}}\int_{0}^{x}\left(  (1-t^{2})f^{\prime\prime
}(t)\right)  ^{\prime}g^{\prime}(t)dt \label{(1) Strauss}%
\end{equation}
exists and is finite. Integration by parts shows that%
\[
\int_{0}^{x}\left(  (1-t^{2})f^{\prime\prime}(t)\right)  ^{\prime}g^{\prime
}(t)dt=(1-t^{2})f^{\prime\prime}(t)g^{\prime}(t)\mid_{0}^{x}-\int_{0}%
^{x}(1-t^{2})f^{\prime\prime}(t)g^{\prime\prime}(t)dt.
\]
Since $(1-x^{2})^{1/2}f^{\prime\prime}(x)$ and $(1-x^{2})^{1/2}g^{\prime
\prime}(x)\in L^{2}(-1,1)$, we see that%
\begin{equation}
(1-x^{2})f^{\prime\prime}(x)g^{\prime\prime}(x)\in L^{1}(-1,1),
\label{L^1 condition}%
\end{equation}
so%
\[
\lim_{x\rightarrow1^{-}}\int_{0}^{x}(1-t^{2})f^{\prime\prime}(t)g^{\prime
\prime}(t)dt
\]
exists and is finite. It follows that%
\[
\lim_{x\rightarrow1^{-}}(1-x^{2})f^{\prime\prime}(x)g^{\prime}(x)
\]
exists and is finite. Suppose%
\[
\lim_{x\rightarrow1^{-}}(1-x^{2})f^{\prime\prime}(x)g^{\prime}(x)=:2c
\]
where we assume that $c\neq0.$ Without loss of generality, assume $c>0.$ Then
there exists $x_{0}\in\lbrack0,1)$ such that%
\begin{equation}
(1-x^{2})f^{\prime\prime}(x)g^{\prime}(x)\geq c\text{ and}
\label{dagger Strauss}%
\end{equation}%
\[
f^{\prime\prime}(x)>0,g^{\prime}(x)>0\text{ \ \ \ \ }(x\in\lbrack x_{0},1)),
\]
implying%
\[
(1-x^{2})f^{\prime\prime}(x)\left\vert g^{\prime\prime}(x)\right\vert \geq
c\frac{\left\vert g^{\prime\prime}(x)\right\vert }{g^{\prime}(x)}\text{
\ \ \ \ }(x\in\lbrack x_{0},1)).
\]
Hence,%
\begin{align}
\int_{x_{0}}^{x}(1-t^{2})f^{\prime\prime}(t)\left\vert g^{\prime\prime
}(t)\right\vert dt  &  \geq c\int_{x_{0}}^{x}\frac{\left\vert g^{\prime\prime
}(t)\right\vert }{g^{\prime}(t)}dt\nonumber\\
&  \geq c\left\vert \int_{x_{0}}^{x}\frac{g^{\prime\prime}(t)}{g^{\prime}%
(t)}dt\right\vert =c\left\vert \ln\left(  g^{\prime}(x)\right)  \right\vert
-c_{1}\text{ \ \ \ \ }(x\in\lbrack x_{0},1)). \label{star Strauss}%
\end{align}
Therefore,%
\[
\lim\sup_{x\rightarrow1^{-}}\left\vert \ln\left(  g^{\prime}(x)\right)
\right\vert <\infty.
\]
\underline{Claim}: There exist positive constants $M_{1},M_{2}$ such that%
\begin{equation}
M_{1}<g^{\prime}(x)<M_{2}\text{ \ \ \ \ }(x\in\lbrack x_{0},1)).
\label{bounded g}%
\end{equation}
Otherwise, if $g^{\prime}(x)$ is unbounded above, there exists a sequence
$\left\{  x_{n}\right\}  _{n\geq1}\subset\lbrack x_{0},1)$ such that%
\[
g^{\prime}(x_{n})\rightarrow\infty\text{ as }n\rightarrow\infty.
\]
It follows from (\ref{star Strauss}) that%
\[
(1-x^{2})f^{\prime\prime}(x)g^{\prime\prime}(x)\notin L^{1}(x_{0},1),
\]
contradicting (\ref{L^1 condition}); hence, $M_{2}>0$ exists as claimed. If an
$M_{1}$, satisfying (\ref{bounded g}), does not exist, then there exists a
sequence $\left\{  y_{n}\right\}  \subset\lbrack x_{0},1)$ such that%
\[
g^{\prime}(y_{n})\rightarrow0\text{ as }n\rightarrow\infty.
\]
Again, it follows from (\ref{star Strauss}) that
\[
(1-x^{2})f^{\prime\prime}(x)g^{\prime\prime}(x)\notin L^{1}(-1,1),
\]
again contradicting (\ref{L^1 condition}). From the claim, it now follows from
(\ref{dagger Strauss}) that%
\[
(1-x^{2})f^{\prime\prime}(x)\geq\frac{c}{g^{\prime}(x)}>\frac{c}{M_{2}%
}=:\widetilde{c}\text{ \ \ \ \ }(x\in\lbrack x_{0},1)).
\]
Consequently,
\[
(1-x^{2})\left(  f^{\prime\prime}(x)\right)  ^{2}>\frac{\widetilde{c}^{\text{
}2}}{1-x^{2}}\text{ \ \ \ \ }(x\in\lbrack x_{0},1)).
\]
Integrating over $[x_{0},1)$ and using the fact that $(1-x^{2})^{1/2}%
f^{\prime\prime}(x)\in L^{2}(-1,1),$ we see that%
\[
\infty>\int_{x_{0}}^{1}(1-t^{2})(f^{\prime\prime}(t))^{2}dt>\widetilde{c}%
^{\text{ }2}\int_{x_{0}}^{1}\frac{dt}{1-t^{2}}=\infty,
\]
a contradiction unless $\widetilde{c}=c=0.$ This completes the proof.
\end{proof}

\begin{theorem}
$T_{1}$ is symmetric in $\left(  W_{1},\phi(\cdot,\cdot)\right)  $.
\end{theorem}

\begin{proof}
Since $T_{1}$ has the Jacobi polynomials $\left\{  P_{n}^{(-1,-1)}\right\}
_{n=2}^{\infty}$ as a complete set of eigenfunctions (see Remark
\ref{Jacobi polynomials and W_1}), it suffices to show that $T_{1}$ is
Hermitian. Let $f,g\in\mathcal{D}(T_{1})=V_{3}.$ Since $V_{3}\subset V_{1}$
and $T_{1}f,T_{1}g\in V_{1}$, we see that%
\[
f(\pm1)=g(\pm1)=0=T_{1}f(\pm1)=T_{1}g(\pm1).
\]
Hence,%
\begin{align*}
\phi\left(  T_{1}f,g\right)   &  =\int_{-1}^{1}\left(  T_{1}f\right)
^{\prime}(x)\overline{g}^{\prime}(x)dx\\
&  =\int_{-1}^{1}\left[  -\left(  (1-x^{2})f^{\prime\prime}(x)\right)
^{\prime}+kf^{\prime}(x)\right]  \overline{g}^{\prime}(x)dx\\
&  =-(1-x^{2})f^{\prime\prime}(x)\overline{g}^{\prime}(x)\mid_{-1}^{1}%
+\int_{-1}^{1}\left[  (1-x^{2})f^{\prime\prime}(x)\overline{g}^{\prime\prime
}(x)+kf^{\prime}(x)\overline{g}^{\prime}(x)\right]  dx\\
&  =\int_{-1}^{1}\left[  (1-x^{2})f^{\prime\prime}(x)\overline{g}%
^{\prime\prime}(x)+kf^{\prime}(x)\overline{g}^{\prime}(x)\right]  dx
\end{align*}
since $-(1-x^{2})f^{\prime\prime}(x)\overline{g}^{\prime}(x)\mid_{-1}^{1}=0$
by Theorem \ref{(4) Hermite}. A similar calculation shows that
\begin{align*}
\phi\left(  f,T_{1}g\right)   &  =\int_{-1}^{1}\left[  -\left(  (1-x^{2}%
)\overline{g}^{\prime\prime}(x)\right)  ^{\prime}+k\overline{g}^{\prime
}(x)\right]  f^{\prime}(x)dx\\
&  =\int_{-1}^{1}\left[  (1-x^{2})f^{\prime\prime}(x)\overline{g}%
^{\prime\prime}(x)+kf^{\prime}(x)\overline{g}^{\prime}(x)\right]  dx.
\end{align*}
Hence $\phi\left(  f,T_{1}g\right)  =\phi\left(  T_{1}f,g\right)  ;$ that is,
$T_{1}$ is symmetric in $\left(  W_{1},\phi(\cdot,\cdot)\right)  .$
\end{proof}

\begin{theorem}
The operator $T_{1}$ has the following properties:
\end{theorem}

\begin{enumerate}
\item[(i)] $T_{1}$ is self-adjoint in $\left(  W_{1},\phi(\cdot,\cdot)\right)
;$

\item[(ii)] $\sigma(T_{1})=\left\{  n(n-1)+k\mid n\geq2\right\}  ;$

\item[(iii)] $\left\{  P_{n}^{(-1,-1)}\right\}  _{n=2}^{\infty}$ is a complete
orthonormal set of eigenfunctions of $T_{1}$ in $\left(  W_{1},\phi
(\cdot,\cdot)\right)  .$
\end{enumerate}

\begin{proof}
Part (iii) is established in Remark \ref{Jacobi polynomials and W_1}. Since it
is well known (for example, see \cite[Theorem 3, p. 373 and Theorem 6, p.
184]{Hellwig}) that a closed, symmetric operator with a complete set of
eigenfunctions is self-adjoint, it suffices, in order to establish (i), to
show $T_{1}$ is closed. To this end, let$\left\{  f_{n}\right\}
\subseteq\mathcal{D}(T_{1})=V_{3}$ such that
\begin{align*}
f_{n}  &  \rightarrow f\text{ \ \ \ \ in }\left(  W_{1},\phi(\cdot
,\cdot)\right) \\
T_{1}f_{n}  &  \rightarrow g\text{ \ \ \ \ in }\left(  W_{1},\phi(\cdot
,\cdot)\right)  .
\end{align*}
We show that $f\in\mathcal{D}(T_{1})$ and $T_{1}f=g.$ Since, by Theorem
\ref{Equivalent Inner Products}, $\phi(\cdot,\cdot)$ and $(\cdot,\cdot)_{1}$
are equivalent, there exist positive constants $c_{1}$ and $c_{2}$ such that%
\[
c_{1}\left\Vert f\right\Vert _{\phi}\leq\left\Vert f\right\Vert _{1}\leq
c_{2}\left\Vert f\right\Vert _{\phi}\text{ \ \ \ \ }(f\in W_{1}=V_{1}).
\]
Hence,%
\[
\left\Vert f_{n}-f\right\Vert _{1}\leq c_{2}\left\Vert f_{n}-f\right\Vert
_{\phi}\rightarrow0;
\]
in particular,
\[
f_{n}\rightarrow f\text{ \ \ \ \ in }\left(  W_{1},(\cdot,\cdot)_{1}\right)
.
\]
Similarly,
\[
\left\Vert T_{1}f_{n}-g\right\Vert _{1}\leq c_{2}\left\Vert T_{1}%
f_{n}-g\right\Vert _{\phi}\rightarrow0
\]
so%
\[
T_{1}f_{n}\rightarrow g\text{ \ \ \ \ in }\left(  W_{1},(\cdot,\cdot
)_{1}\right)  .
\]
Since $T_{1}$ is self-adjoint in $\left(  W_{1},(\cdot,\cdot)_{1}\right)  $,
it is closed implying that $f\in\mathcal{D}(T_{1})$ and $T_{1}f=g.$ Also, we
know that, for $n\geq2,$%
\[
T_{1}P_{n}^{(-1,-1)}=\ell\lbrack P_{n}^{(-1,-1)}]=(n(n-1)+k)P_{n}^{(-1,-1)}.
\]
This implies%
\[
\left\{  n(n-1)+k\mid n\geq2\right\}  \subseteq\sigma(T_{1}).
\]
However, from the completeness of $\left\{  P_{n}^{(-1,-1)}\right\}
_{n=2}^{\infty}$ and since $\lambda_{n}:=n(n-1)+k\rightarrow\infty$, it
follows from well-known results that
\[
\sigma(T_{1})=\left\{  n(n-1)+k\mid n\geq2\right\}  ,
\]
which proves (ii).
\end{proof}

Next, we define the operator $T_{2}:\mathcal{D}(T_{2})\subset W_{2}\rightarrow
W_{2}$ by%
\begin{align*}
(T_{2}f)(x)  &  =\ell\lbrack f](x)\\
\mathcal{D}(T_{2})  &  :=W_{2}.
\end{align*}
It is straightforward to check that $T_{2}$ is symmetric in $W_{2}$ and, since
$\mathcal{D}(T_{2})=W_{2}$, it follows that $T_{2}$ is self-adjoint.

We now construct the self-adjoint operator $T$ in $\left(  W,\phi(\cdot
,\cdot)\right)  $, generated by the Jacobi differential expression
$\ell\lbrack\cdot],$ which has the \textit{entire} set of Jacobi polynomials
$\left\{  P_{n}^{(-1,-1)}\right\}  _{n=0}^{\infty}$ as eigenfunctions and
spectrum $\sigma(T)=\left\{  n(n-1)+k\mid n\in%
\mathbb{N}
_{0}\right\}  .$

We define the domain of this operator $T$ to be
\[
\mathcal{D}(T):=\mathcal{D}(T_{1})\oplus\mathcal{D}(T_{2})=V_{3}\oplus W_{2}.
\]
Then each $f\in\mathcal{D}(T)$ can be written as $f=f_{1}+f_{2},$ where
$f_{i}\in\mathcal{D}(T_{i})$ ($i=1,2$). Define $T:\mathcal{D}(T)\subset
W\rightarrow W$ by%
\[
Tf:=T_{1}f_{1}+T_{2}f_{2}=\ell\lbrack f_{1}]+\ell\lbrack f_{2}]=\ell\lbrack
f].
\]
A proof that operators of this form are self-adjoint can be found in
\cite[Theorem 11.1]{Everitt-Littlejohn-Wellman}. Furthermore, since we know
explicitly the domains of $T_{1}$ and $T_{2}$, we can specifically determine
the domain $\mathcal{D}(T)$ of $T.$

\begin{theorem}
$T$ is self-adjoint in $\left(  W,\phi(\cdot,\cdot)\right)  $ and has domain%
\begin{align*}
\mathcal{D}(T) &  =\{f:[-1,1]\rightarrow%
\mathbb{C}
\mid f\in AC[-1,1];f^{\prime},f^{\prime\prime}\in AC_{\mathrm{loc}}(-1,1);\\
&  \hspace{1.5in}(1-x^{2})f^{\prime\prime\prime},(1-x^{2})^{1/2}%
f^{\prime\prime},f^{\prime}\in L^{2}(-1,1)\}.
\end{align*}
Furthermore, $\sigma(T)=\left\{  n(n-1)+k\mid n\in%
\mathbb{N}
_{0}\right\}  $ and has the Jacobi polynomials $\left\{  \widetilde{P}%
_{n}^{(-1,-1)}\right\}  _{n=0}^{\infty}$ as a complete set of eigenfunctions.
\end{theorem}

\begin{proof}
Define%
\begin{align*}
\mathcal{D}  &  :=\{f:[-1,1]\rightarrow%
\mathbb{C}
\mid f\in AC[-1,1];f^{\prime},f^{\prime\prime}\in AC_{\mathrm{loc}}(-1,1);\\
&  \hspace{1.5in}(1-x^{2})f^{\prime\prime\prime},(1-x^{2})^{1/2}%
f^{\prime\prime},f^{\prime}\in L^{2}(-1,1)\}.
\end{align*}
Since $\mathcal{D}(T_{i})\subset\mathcal{D}$ for $i=1,2,$ it is clear that
$\mathcal{D}(T)=\mathcal{D}(T_{1})\oplus\mathcal{D}(T_{2})\subset\mathcal{D}.$
Conversely, let $f\in\mathcal{D}$. Writing $f=f_{1}+f_{2}$ where each $f_{i}$
$(i=1,2)$ is given as in $($\ref{Representation1}$),$ we see that $f\in$
$\mathcal{D}(T).$ The proof of the last statement in the theorem is clear.
\end{proof}

Using Theorem \ref{CHEL}, we can further refine the domain of $T;$ we leave
the details to the reader.

\begin{corollary}
The domain of $T$ is given by
\[
\mathcal{D}(T)=\{f:[-1,1]\rightarrow%
\mathbb{C}
\mid f\in AC[-1,1];f^{\prime},f^{\prime\prime}\in AC_{\mathrm{loc}%
}(-1,1;);(1-x^{2})f^{\prime\prime\prime}\in L^{2}(-1,1)\}.
\]

\end{corollary}

\end{document}